\documentclass[12pt, reqno]{amsart}
\usepackage{amsmath,amssymb,amsthm, amscd, amsfonts, graphicx, tabularx, array, CJK, fancyhdr, mathrsfs}
\usepackage{stmaryrd}
\usepackage{txfonts}
\usepackage{bbm}
\usepackage{verbatim}

\usepackage{color}

\usepackage[arc,all]{xy}

\theoremstyle{plain}
\numberwithin{equation}{section}

\newtheorem{thm}{Theorem}[section]
\newtheorem{prop}[thm]{Proposition}
\newtheorem{lemma}[thm]{Lemma}
\newtheorem{corollary}[thm]{Corollary}
\newtheorem{conj}[thm]{Conjecture}
\theoremstyle{remark}
\newtheorem{remark}[thm]{Remark}
\theoremstyle{definition}
\newtheorem{defn}[thm]{Definition}
\newtheorem{example}[thm]{Example}

\if
\theoremstyle{plain}
\numberwithin{equation}{chapter}

\newtheorem{thm}{Theorem}[chapter]
\newtheorem{prop}[thm]{Proposition}
\newtheorem{lemma}[thm]{Lemma}

\theoremstyle{definition}

\theoremstyle{remark}
\newtheorem{remark}[thm]{Remark}
\fi

\renewenvironment{proof}{{\noindent \textbf{Proof:} }}{\hfill $\Box$\medskip}

\def\fC{\mathbb{C}} 

\def\fZ{\mathbb{Z}}
\def\fQ{\mathbb{Q}}

\def\[[{\llbracket}
\def\]]{\rrbracket}

\def\Dbar{\leavevmode\lower.6ex\hbox to 0pt{\hskip-.23ex
\accent"16\hss}D}
\def\dbar{\leavevmode\hbox to 0pt{\hskip.2ex
\accent"16\hss}d}

\def\polhk#1{\setbox0=\hbox{#1}{\ooalign{\hidewidth
\lower1.5ex\hbox{`}\hidewidth\crcr\unhbox0}}}
\def\cftil#1{\ifmmode\setbox7\hbox{$\accent"5E#1$}\else
\setbox7\hbox{\accent"5E#1}\penalty 10000\relax\fi\raise 1\ht7
\hbox{\lower1.15ex\hbox to 1\wd7{\hss\accent"7E\hss}}\penalty 10000
\hskip-1\wd7\penalty 10000\box7}
\def\cfudot#1{\ifmmode\setbox7\hbox{$\accent"5E#1$}\else
\setbox7\hbox{\accent"5E#1}\penalty 10000\relax\fi\raise 1\ht7
\hbox{\raise.1ex\hbox to 1\wd7{\hss.\hss}}\penalty 10000
\hskip-1\wd7\penalty 10000\box7}

\def\<{\prec} 
\def\>{\succ}

\begin{document}


\title{The double of representations of Cohomological Hall Algebra for $A_1$-quiver}

\author{Xinli Xiao}

\date{\today}

\address{Mathematical Department, Kansas State University,
Cardwell Hall 128, Manhattan, Kansas, 66502} \email{xiaoxl@math.ksu.edu}


\baselineskip=18pt

\begin{abstract}
We compute two representations of COHA for $A_1$-quiver. The two untwisted representations can be combined into a representation of $D_{n+1}$ Lie algebra. The untwisted increasing representation and the twisted decreasing representation can be combined into a representation of a finite Clifford algebra.
\end{abstract}

\maketitle

\topmargin -0.1in
\renewcommand{\thepage}{\arabic{page}}

\bibliographystyle{mybst}


\section{introduction}

The aim of this paper is to define and discuss two representations of the Cohomological Hall algebras, and combine them into a single representation of the algebra which is called ``full'' (or ``double'') COHA in \cite{Soi2014}.

Cohomological Hall algebra (COHA for short) was introduced in \cite{KoS2011}. The definition is similar to the definition of conventional Hall algebra (see e.g. \cite{Sch2006}) or its motivic version (see e.g. \cite{KoS2008}). Instead of spaces of constructible functions on the stack of objects of an abelian category, one considers cohomology groups of the stacks. The product is defined through the pullback-pushforward construction. Details can be found in \cite{KoS2011}.

By analogy with conventional Hall algebra of a quiver, which gives the ``positive'' part of a quantization of the corresponding Lie algebra, one may want to define the ``double'' COHA, for which the one defined in \cite{KoS2011} would be a ``positive part''. Following the discussion in \cite{Soi2014}, we study the double of representations of COHA, and hope to find the double of COHA through its representations.

This paper focuses on $A_1$-quiver. Stable framed representations of the quiver are used to produce two representations of COHA. Since the moduli spaces of stable framed representations of $A_1$-quiver are Grassmannians, we actually define two representations on the cohomology of Grassmannians. We show that the operators from these two representations form $D_{n+1}$-Lie algebra. We also make a modification to the decreasing representation and form a twisted decreasing representation. The operators from untwisted increasing operators and twisted decreasing operators form a finite Clifford algebra. These confirm the conjecture from \cite{Soi2014} that the double of $A_1$-COHA is the infinite Clifford algebra.


\section{Two geometric representations of $A_1$-COHA}

\subsection{COHA}

Let $Q$ be a quiver with $N$ vertices. Given a dimension vector $\gamma=(\gamma_i)_{i=1}^N$, $ M_{\gamma}$ is the space of complex representations with fixed underlying vector space $\bigoplus_{i=1}^{N} \fC^{\gamma_i}$ of dimension vector $\gamma$, and $G_{\gamma}=\prod_{i=1}^NGL_{\gamma_i}(\fC)$ is the associated gauge group. $[ M_{\gamma}/G_{\gamma}]$ is the stack of representations of $Q$ with fixed dimension vector $\gamma$. As a vector space, COHA of $Q$ is defined to be $\mathcal H:=\bigoplus_{\gamma}\mathcal H_{\gamma}:=\bigoplus_{\gamma}H^*([M_{\gamma}/G_{\gamma}]):=\bigoplus_{\gamma}H^*_{G_{\gamma}}( M_{\gamma})$. Here by equivariant cohomology of a complex algebraic variety $M_{\gamma}$ acted by a complex algebraic group $G_{\gamma}$ we mean the usual (Betti) cohomology with coefficients in $\fQ$ of the bundle $EG_{\gamma}\times_{G_{\gamma}}M_{\gamma}$ associated to the universal $G_{\gamma}$-bundle $EG_{\gamma}\rightarrow BG_{\gamma}$ over the classifying space of $G_{\gamma}$. The product $*:\mathcal H\otimes \mathcal H\rightarrow \mathcal H$ is defined by means of the pullback-pushforward construction in \cite{KoS2011}.

\subsection{$A_1$-COHA}\label{incre_basis}
Let $Q$ be $A_1$. $N=1$. Since there is only one representation with fixed underlying vector space $\fC^{d}$ of dimension $d$, $M_{d}$ is a point and $G_{d}=GL_d(\fC)$. Therefore $\mathcal H_d=H^*_{GL_d(\fC)}(  M_d)=\fQ[x_{1,d},\ldots,x_{d,d}]^{S_d}$ is the algebra of symmetric polynomials in variables $x_{1,d},\ldots,x_{d,d}$. It is possible to talk about the geometric interpretation of these variables. They can be treated as the first Chern classes of the tautological bundles over the classifying space of $G_d$. For details see e.g. \cite{Xia2013}.

The COHA $\mathcal H$ for quiver $A_1$ is described in \cite{KoS2011}. It is the infinite exterior algebra generated by odd elements $\phi_0, \phi_1,\phi_2\ldots$ with wedge $\wedge$ as its product. Generators $(\phi_i)_{i\geq 0}$ correspond to the additive generators $(x^i_{1,1})_{i\geq0}$ of $\mathcal H_1=\fQ[x_{1,1}]$. A monomial in the exterior algebra
$$\phi_{k_1}\wedge\ldots\wedge\phi_{k_d}\in\mathcal H_d,\quad 0\leq k_1<\ldots<k_d$$
corresponds to the Schur symmetric polynomial $s_{\lambda}(x_{1,d},\ldots,x_{d,d})$, where $\lambda=(\lambda_d,\ldots,\lambda_1)=(k_d-d+1,k_{d-1}-d+2,\ldots,k_1)$ is a partition.

Let $\Phi_{\bf k}=\phi_{k_1}\wedge\ldots\wedge\phi_{k_d}$ with index ${\bf k}=(k_1,\ldots,k_d)$, $0\leq k_1<\ldots<k_d$. Denote by ${\bf k}(\lambda)$ the index related to the partition $\lambda$ and by $\lambda({\bf k})$ the partition related to the basis index $\bf k$. Then we have $\Phi_{{\bf k}(\lambda)}=s_{\lambda(\bf k)}$.

\subsection{Stable framed representations}\label{section23}
Fix a dimension vector ${\bf n}=(n_i)_{i=1}^N$. A {\it framed representation} of $Q$ of dimension vector $\gamma$ is a pair $(V,f)$, where $V$ is an ordinary representation of $Q$ of dimension $\gamma$ and $f=(f_i)_{i=1}^{N}$ is a collection of linear maps from $\fC^{n_i}$ to $V_i$. The set of framed representations of dimension vector $\gamma$ with framed structure dimension vector $\bf n$ is denoted by ${\hat M}_{\gamma,{\bf n}}$. It carries a natural gauge group $G_{\gamma}$-action. See e.g. \cite{Rei2008a}.

For the notion of stable framed representation of a quiver, see e.g. \cite{Rei2008} (more general framework of triangulated categories can be found in \cite{Soi2014}). We focus on the trivial stability condition. In this case, a framed representation is called {\it stable} if there is no proper (ordinary) subrepresentation of $V$ which contains the image of $f$. The set of stable framed representations of dimension vector $\gamma$ with framed structure dimension vector $\bf n$ is denoted by $  {\hat M}^{st}_{\gamma,{\bf n}}$. The gauge group $G_{\gamma}$ of $ M_{\gamma,{\bf n}}$ induces a $G_{\gamma}$-action on $\hat M^{st}_{\gamma,{\bf n}}$. The stack of stable framed representations $[\hat M^{st}_{\gamma,{\bf n}}/G_{\gamma}]$ is in fact a smooth projective scheme. We denote it by $\mathcal M^{st}_{\gamma,{\bf n}}$ and call it {\it the smooth model} of quiver $Q$ with dimension $\gamma$ and framed structure $\bf n$.

The pullback-pushforward construction is applied to the cohomology of the scheme of stable framed representations. This construction leads to two representations of COHA for the quiver $Q$ which we describe below.

Fix two dimension vectors $\gamma_1$ and ${\gamma_2}$. Set $\gamma=\gamma_1+\gamma_2$. Consider the scheme consisting of diagrams
\begin{equation}
\mathcal M^{st}_{\gamma_2,\gamma,{\bf n}}:=\{\xymatrix{0\ar[r]{}&E_1\ar[r]&E\ar[r]&E_2\ar[r]&0\\&&\fC^{\bf n}\ar[u]^{f}\ar[ur]_{f_2}&&}\},
\end{equation}
where $E_1\in  M_{\gamma_1}$, $(E,f)\in \mathcal M^{st}_{\gamma,{\bf n}}$, $(E_2,f_2)\in \mathcal M^{st}_{\gamma_2,{\bf n}}$. $f:\fC^{\bf n}\rightarrow E$ and $f_2:\fC^{\bf n}\rightarrow E_2$ are the framed structures attached to $E$ and $E_2$ respectively. The subgroup of the automorphism group of $E$ which preserves the embedding of $E_1$ is denoted by $P_{\gamma_1,\gamma,{\bf n}}$. It plays the role of the automorphism group of $\mathcal M^{st}_{\gamma_2,\gamma,{\bf n}}$. The natural projections from the diagram to its components give the following diagram:
\begin{equation}\label{corres}
  \xymatrix{&\mathcal M^{st}_{\gamma,{\bf n}}&\\&[\hat M^{st}_{\gamma_2,\gamma,{\bf n}}/P_{\gamma_2,\gamma,{\bf n}}]\ar[u]^{p}\ar[dr]^{p_2}\ar[dl]_{p_1}&\\[  M_{\gamma_1}/G_{\gamma_1}]&&\mathcal M^{st}_{\gamma_2,{\bf n}}}.
\end{equation}
The map $p_*(p_1^*(\phi_1)\cup p_2^*(\varphi_2))$ defines a morphism from $H^*({\mathcal M}^{st}_{\gamma_2,{\bf n}})$ to $H^*({\mathcal M}^{st}_{\gamma,{\bf n}})$ for $\phi_1\in \mathcal H_{\gamma_1}$ and $\varphi_2\in H^*(\mathcal M^{st}_{\gamma_2,{\bf n}})$. This morphism induces a representation of $\mathcal H=\bigoplus_{\gamma}\mathcal H_{\gamma}$ on $\bigoplus_{\gamma}H^*(\mathcal M^{st}_{\gamma,{\bf n}})$. It is called {\it the increasing representation} of COHA for the quiver $Q$, and denoted by $R^+_{\bf n}$. Similarly, the map $(p_2)_*(p_1^*(\phi_1)\cup p^*(\varphi))$ for $\phi_1\in \mathcal H_{\gamma_1}$ and $\varphi\in H^*(\mathcal M^{st}_{\gamma,{\bf n}})$ gives {\it the decreasing representation} $R^-_{\bf n}$ on the cohomology of the smooth model. In order to have well-defined representations one needs to show that $p$ and $p_2$ are proper morphisms. For $A_1$-case the properness is obvious (see Section \ref{2repn_A1} below).


\subsection{$A_1$-case} 

 Let $n$ be the framed structure dimension. A framed representation $(\fC^d,f)$ of $A_1$-quiver is stable if and only if $f:\fC^n\rightarrow \fC^d$ is surjective. Thus the stable framed moduli space $\mathcal M^{st}_{d,n}$ is the Grassmannian (of quotient spaces) $Gr(d,n)$ for $0\leq d\leq n$, and empty for $d>n$. 

It is well known (see e.g. \cite{Ful1997}, p.$161$) that the cohomology of full flag variety $Fl(n)$ is isomorphic to $R(n)=\fQ[x_1,\ldots,x_n]/(e_1(x_1,\ldots,x_n),\ldots,e_n(x_1,\ldots,x_n))$, where $e_i(x_1,\ldots,x_n)$ represents the $i$-th elementary symmetric polynomial. The cohomology of Grassmannian $Gr(d,n)$ is a subalgebra of $R(n)$ which is generated by Schur polynomials in variables $x_1,\ldots,x_d$. Thus we can use $s_{\lambda}(x_1,\ldots,x_d)$ to represent classes in $H^*(Gr(d,n))$. There is a natural projection $\pi: Fl(n)\rightarrow Gr(d,n)$. By abuse of notations, we use the same symbol $x_i$ to denote the classes in $Gr(d,n)$ whose pullback $\pi^*(x_i)$ is $x_i\in H^*(Fl(n))$.

Classes in $H^*(Gr(d,n))$ have an alternative presentation.

\begin{lemma}\label{Lemma2.1}
  In $H^*(Gr(d,n))$, $s_{\lambda}(x_1,\ldots,x_d)=(-1)^{|\lambda|}s_{\lambda'}(x_{d+1},\ldots,x_n)$, where $\lambda'$ is the transpose partition of $\lambda$.
\end{lemma}
\begin{proof}
 The above identity can be easily deduced from the identity $\prod_{i=1}^d\frac{1}{1-x_it}=\prod_{i=d+1}^n(1-x_it)$ (see e.g. \cite{Ful1997}, p.$163$) in the ring $R(n)[t]$.

 Since
 \begin{equation}
   \prod_{i=1}^d\frac{1}{1-x_it}=\sum_{r\geq0}h_r(x_1,\ldots,x_d)t^r
 \end{equation}
 and
  \begin{equation}
   \prod_{i=d+1}^n(1-x_it)=\sum_{r\geq0}e_r(x_{d+1},\ldots,x_n)(-t)^r
 \end{equation}
 where $h_r$ (resp. $e_r$) stands for the $r$-th complete symmetric polynomial (resp. elementary symmetric polynomial), we have
\begin{equation}
  h_r(x_1,\ldots,x_d)=(-1)^re_r(x_{d+1},\ldots,x_n), \quad r\geq0.
\end{equation}

By Jacobi-Trudi identity,
\begin{eqnarray*}
  s_{\lambda'}(x_1,\ldots,x_d)&=&det(e_{\lambda_i-i+j}(x_1,\ldots,x_d))\\
  &=&det((-1)^{\lambda_i-i+j}h_{\lambda_i-i+j}(x_{d+1},\ldots,x_n))\\
  &=&(-1)^{|\lambda|}det(h_{\lambda_i-i+j}(x_{d+1},\ldots,x_n))\\
  &=&(-1)^{|\lambda|}s_{\lambda}(x_{d+1},\ldots,x_n).
\end{eqnarray*}
The third identity comes from the fact that
\begin{eqnarray*}
 det(h_{\lambda_i-i+j}t^{\lambda_i-i+j})&=&\sum_{\omega}\sum_{i=1}^n(-1)^{\omega}h_{\lambda_i-i+\omega(i)}t^{\lambda_i-i+\omega(i)}\\
 &=&\sum_{\omega}t^{\sum_{i=1}^n\lambda_i-i+\omega(i)}\sum_{i=1}^n(-1)^{\omega}h_{\lambda_i-i+\omega(i)}\\
 &=&\sum_{\omega}t^{|\lambda|}\sum_{i=1}^n(-1)^{\omega}h_{\lambda_i-i+\omega(i)}\\
 &=&t^{|\lambda|}det(h_{\lambda_i-i+j}).
\end{eqnarray*}

\end{proof}

In the following we will use this ``transpose'' presentation to do some computations.

\subsection{Two representations of $A_1$-COHA}\label{2repn_A1}

 The scheme $[M^{st}_{d_2,d,n}/P_{d_2,d,n}]$ in $A_1$-quiver case is isomorphic as a scheme to the two-step flag variety $F_{d_2,d,n}$, which is variety of the flags $\{\fC^n\twoheadrightarrow \fC^d\twoheadrightarrow \fC^{d_2}\}$. Let $\phi_i$ be a generator of $\mathcal H_1$, and $s_{\lambda}$ be the Schur polynomial considered as an element of the cohomology of the Grassmannian $H^*(Gr(d_2,n))$ whose partition is $\lambda$. In this case, $p$ is the obvious projection from $F_{d_2,d,n}$ to $Gr(d,n)$ and $p_2$ is the obvious projection from $F_{d_2,d,n}$ to $Gr(d_2,n)$. Therefore both $p$ and $p_2$ are proper morphisms of stacks (which are in fact schemes), and the increasing and decreasing representations introduced in Section \ref{section23} are well defined. 

 Now we want to compute the increasing representation by the formula $p_*(p_1^*(\phi_i)\cup p_2^*(s_{\lambda}))$. Note that in this case, $d_1=1$. 
 Recall that $\phi_i$ represents the polynomial $\phi_i(x_{1,1})=x_{1,1}^i$. 
 Using the geometric interpretation, $x^i_{1,1}$ is treated as the first Chern class of the tautological line bundle $\mathscr O(-i)$ over the classifying space of $G_1$. $\mathscr O(-i)$ will be pulled back through $p_1$ to the line bundle over $F_{d_2,d,n}$ associated to the corresponding character of $G_{d_1}$ when treating $G_{d_1}$ as a subquotient of $P_{d_2,d,n}$. Hence $p_1^*(\phi_i)$ will be the first Chern class of the line bundle described above, which is $\phi_i(x_{d_2+1})=x^i_{d_2+1}$.


  As homogenous spaces, $Gr(d,n)\approx GL_n(\fC)/P_{d,n}$, $Gr(d_2,n)\approx GL_n(\fC)/P_{d_2,n}$ and $F(d_2,d,n)\approx GL_n(\fC)/P_{d_2,d,n}$. We use the formula in \cite{Bri1996} to compute the pushforward.

\begin{thm}\label{thm2.5}\cite{Bri1996}. 
Let $G$ be a connected reductive algebraic group over $\fC$ and $B$ a Borel subgroup. Choose a maximal torus $T\subset B$ with Weyl group $W$. The set of all positive roots of the root system of $(G,T)$ is denoted by $\Delta^+$. Let $P\supset B$ be a parabolic subgroup of $G$, with the set of positive roots $\Delta^+(P)$ and Weyl group $W_P$. Let $L_{\alpha}$ be the complex line bundle over $G/B$ which is associated to the root $\alpha$. The Gysin homomorphism $f_*:H^*(G/B)\rightarrow H^*(G/P)$ is given by
\begin{equation}
  f_*(p)=\sum_{w\in W/W_P}w\cdot\frac{p}{\prod_{\alpha\in\Delta^+\backslash\Delta^+(P)}c_1(L_{\alpha})}.
\end{equation}
\end{thm}

Applying Thm \ref{thm2.5}, for $s_{\lambda}\in H^{*}(Gr(d_2,n))$,
\begin{equation}\label{incre-formula}
  (\phi_i^+\cdot s_{\lambda})(x_1,\ldots,x_{d_2+1})=\sum_{i_1<\ldots<i_{d_2}}\frac{s_{\lambda}(x_{i_1},\ldots,x_{i_{d_2}})\phi_i(x_{i_{d_2+1}})}{\prod_{j=1}^{d_2}(x_{i_j}-x_{i_{d_2+1}})}.
\end{equation}

Similarly, the formula of the decreasing actions is
\begin{equation}\label{decre-formula-1}
  (\phi_i^-\cdot s_{\lambda})(x_1,\ldots,x_{d_2-1})=\sum_{i_1<\ldots<i_{d_2}}\frac{s_{\lambda}(x_{i_1},\ldots,x_{i_{d_2}})\phi_i(x_{i_{d_2}})}{\prod_{j=d_2+1}^n(x_{i_{d_2}}-x_{j})}.
\end{equation}

\begin{remark}
In Formula (\ref{decre-formula-1}), variables $x_i$ for $i>d_2-1$ appear on the right side, which do not belong to the variables on the left side. This is not a contradiction because of the formula $s_{\lambda}(x_1,\ldots,x_d)=(-1)^{|\lambda|}s_{\lambda'}(x_{d+1},\ldots,x_n)$ by Lemma \ref{Lemma2.1}. More details will be discussed in the following section.
\end{remark}

\begin{remark}
The construction above actually only defines an incrasing operator $\phi^+_{i,d}$ from $H^*(Gr(d,n))$ to $H^*(Gr(d+1,n))$ and an decreasing operator $\phi^-_{i,d}$ from $H^*(Gr(d,n))$ to $H^*(Gr(d-1,n))$. The increasing operator we need is $\phi^+_i=\sum_{d=0}^n\phi^+_{i,d}$. The decreasing operator we need is $\phi^-_i=\sum_{d=0}^n\phi^-_{i,d}$. We can then define {\it the twisted decreasing operator} by $\hat{\phi}^-_i=\sum_{d=0}^n(-1)^{d-1}\phi^-_{i,d}$. We call the representation formed by these operators {\it the twisted decreasing representation} and denote it by $\hat{R}^-_n$.
\end{remark}



\section{Increasing and decreasing operators}

\subsection{Increasing operators}

The key result of this subsection is adapted from \cite{Fra2013}.
\begin{prop}\cite{Fra2013}.
 The increasing representation structure is induced by the open embedding $j: \hat M^{st}_{d,n}\rightarrow \hat M_{d,n}$. The induced map $j^*:\mathcal H\rightarrow R^+_n$ is $\mathcal H$-linear and surjective. The kernel of $j^*$ equals $\sum_{p\geq0, q>0}\mathcal H_p\wedge(e_q^n\cup \mathcal H_q)$, where $e_q=\prod_{i=1}^dx_i$.
\end{prop}
\begin{proof}
 In \cite{Fra2013}, the similar result for $n=1$ is proved. It can be easily generalized to $n>1$ case for $A_1$-quiver.
\end{proof}

The next lemma follows immediately from the definition of Schur polynomials.
\begin{lemma}
  $s_{(\lambda_d+1,\lambda_{d-1}+1,\ldots,\lambda_1+1)}=e_ds_{\lambda}$ for $s_{\lambda}\in \fQ[x_1,\ldots,x_d]^{S_d}$ and $e_d=\prod_{i=1}^dx_i$. Thus $e^n_d\cup \Phi_{\bf k}=\Phi_{\bf k+n}$ for $\Phi_{\bf k}\in \mathcal H_d$, and ${\bf n}=(n,n,\ldots,n)$.
\end{lemma}

Finally, we come to the result, whose proof is straightforward.

\begin{prop}\label{propincrebasis}
The increasing representation $R^+_n$ is a quotient of $\mathcal H=\bigwedge^*(\mathcal H_1)$ whose kernel is the subalgebra generated by $\{\phi_i\}_{i\geq n}$. Thus $R_n^+$ is isomorphic to $\bigwedge^*(V(n))$ where $V(n)$ is the linear space spanned by $\phi_0, \ldots, \phi_{n-1}$ and the action is given by wedge product from left. Then $\{\phi_{k_1}\wedge\ldots\wedge\phi_{k_d}\}_{k_1<\ldots<k_d}, 0\leq d\leq n-1$ form a basis of $R^+_n$.
\end{prop}

\subsubsection{Two presentations of classes in the cohomology of Grassmannian}\label{dual_1}

Proposition \ref{propincrebasis} implies that we can use the notations introduced in section \ref{incre_basis} to represent cohomology classes of Grassmannians, as well as those in COHA, since they share the same product structure. Thus in $H^*(Gr(d,n))$, $\Phi_{\bf k}=\phi_{k_1}\wedge\ldots\wedge\phi_{k_d}(x_1,\ldots,x_d)$ with index ${\bf k}=(k_1,\ldots,k_d)$ can represent the Schur polynomial $s_{\lambda({\bf k})}(x_{1,d},\ldots,x_{d,d})$, where $0\leq k_1<\ldots<k_d\leq n-1$ and $\lambda=(\lambda_d,\ldots,\lambda_1)=(k_d-d+1,k_{d-1}-d+2,\ldots,k_1)$ is a partition of length $\leq n$. 

Let $\lambda'$ be the transpose partition of $\lambda$, and ${\bf k'}={\bf k}(\lambda')$. By Lemma \ref{Lemma2.1}, $\Phi_{\bf k}(x_1,\ldots,x_d)=(-1)^{|\lambda|}\Phi_{\bf k'}(x_{d+1},\ldots,x_n)$. $\Phi_{\bf k}$ is called {\it the ordinary presentation} of the correspondent class $s_{\lambda}$, and $(-1)^{|\lambda|}\Phi_{\bf k'}$ is called {\it the transpose presentation}.

\subsection{Decreasing operators}

Our goal is to understand the decreasing representation using the basis $\{\Phi_{\bf k}\}_{\bf k}$ of $R^+_n$. From Section \ref{dual_1}, the equation (\ref{decre-formula-1}) can be rewritten as
\begin{equation}\label{decre-formula-ture}
\begin{split}
 (\phi_i^-\cdot \Phi_{\bf k})(x_1,\ldots,x_{d_2-1})&=\sum_{i_1<\ldots<i_{d_2}}\frac{\Phi_{\bf k}(x_{i_1},\ldots,x_{i_{d_2}})\phi_i(x_{i_{d_2}})}{\prod_{j=d_2+1}^n(x_{i_{d_2}}-x_{j})}\\
  &=  (-1)^{|\lambda(\bf k)|}\sum_{i_{d_2+1}<\ldots<i_n}\frac{\Phi_{\bf k'}(x_{i_{d_2+1}},\ldots,x_{i_n})\phi_i(x_{i_{d_2}})}{\prod_{j=d_2+1}^n(x_{i_{d_2}}-x_{i_{j}})}\\
 &=(-1)^{|\lambda|+n-d_2}(\phi_i^+\cdot\Phi_{\bf k'})(x_{d_2},\ldots,x_n).
\end{split}
 \end{equation}
This formula suggests an algorithm. Start from an ordinary presentation of a class $ \Phi_{\bf k}= \phi_{k_1}\wedge\ldots\wedge\phi_{k_d}$ in $H^*(Gr(d,n))$, where ${\bf k}=(k_1,\ldots,k_d)$, and $0\leq k_1<\ldots<k_d\leq n-1$. First we change $\Phi_{\bf k}(x_1,\ldots,x_d)$ to $(-1)^{|\lambda(\bf k)|}\Phi_{\bf k'}(x_{d+1},\ldots,x_n)$ by Lemma \ref{Lemma2.1}. Then apply $\phi_i^-$ to $\Phi_{\bf k'}$ using formula (\ref{decre-formula-ture}) and Proposition \ref{propincrebasis}. Finally change the result back to the ordinary presentation.

We need the following lemma to help us to do these transformations.

\begin{lemma}
 If $\phi_r$ appears in $\Phi_{{\bf k}'(\lambda)}$, $\phi_{n-r-1}$ will not appear in $\Phi_{{\bf k}(\lambda)}$. On the other hand, if $\phi_r$ doesn't appear in $\Phi_{{\bf k}'(\lambda)}$, $\phi_{n-r-1}$ will appear in $\Phi_{{\bf k}(\lambda)}$.

\end{lemma}
\begin{proof}
From Section \ref{propincrebasis}, $\lambda=(\lambda_d,\ldots,\lambda_1)=(k_d-d+1,k_{d-1}-d+2,\ldots,k_1)$ is a partition of length $\leq n$. The transpose partition is defined by $\lambda'_j=\#\{\lambda_i\geq n-d+1-j\}$ for $1\leq j\leq n-d$. Thus we have
\begin{equation}
  \lambda_{d-i+1}=\begin{cases}
    n-d&\text{if}\ 1\leq i\leq \lambda'_1,\\
     n-d-j&\text{if}\ \lambda'_j+1\leq i\leq \lambda'_{j+1}\ \text{for}\ 1\leq j\leq n-d-1,\\
    0&\text{if}\ \lambda'_{n-d}+1\leq i\leq d.
  \end{cases}
\end{equation}

From $\lambda=(\lambda_d,\ldots,\lambda_1)=(k_d-d+1,k_{d-1}-d+2,\ldots,k_1)$, it immediately implies
\begin{equation}
  k_{d-i+1}
  =\begin{cases}
    n-i&\text{if}\ 1\leq i\leq \lambda'_1,\\
      n-i-j&\text{if}\ \lambda'_j+1\leq i\leq \lambda'_{j+1}\ \text{for}\ 1\leq j\leq n-d-1,\\
      d-i&\text{if}\ \lambda'_{n-d}+1\leq i\leq d.
  \end{cases}
\end{equation}
Then $n-k'_{j+1}= n-j-\lambda'_{j+1}\leq k_{d-i+1}=n-i-j\leq n-j-\lambda'_j-1= n-{k'_j}-2$ if $\lambda'_j+1\leq i\leq \lambda'_{j+1}$ for $1\leq j\leq n-d-1$, or $0=d-d\leq k_{d-i+1}=d-i\leq d-\lambda'_{n-d}-1=n-k'_{n-d}-2$ if $\lambda'_{n-d}+1\leq i\leq d$, or $n-k'_1=n-\lambda'_1\leq k_d=n-i\leq n-1$. Therefore $k_{d-i+1}$ would run over all integers between $n-k'_{j+1}$ and $n-k'_j-2$, or between $0$ and $n-k'_{n-d}-2$, or between $n-k'_1$ and $n-1$. If $\phi_r$ doesn't appear in $\Phi_{{\bf k'}(\lambda)}$, there are three cases. If $k'_s< r< k'_{s+1}$ for $1\leq s\leq n-d-1$, $n-k'_{s+1}\leq n-r-1\leq n-k'_s-2$. If $k'_{n-d}<r\leq d$, $0\leq n-r-1\leq n-k'_{n-d}-2$. If $0\leq r<k'_1$, $n-k'_1\leq n-r-1\leq n-1$. This means that there exists some $1\leq i\leq d$ such that $k_{d-i+1}=n-r-1$.

If $\phi_r$ appear in $\Phi_{{\bf k}'(\lambda)}=\phi_{k'_1}\wedge\ldots\wedge\phi_{k'_{n-d}}$, let $r=k'_s$. Then $k_{d-i+1}$ can never be $n-k'_s-1=n-r-1$ for $1\leq i\leq d$.

\end{proof}

\begin{defn}
We introduce the {\it right partial derivative operator} $\partial_i^R: \bigwedge^*(V(n))\rightarrow\bigwedge^*(V(n))$ to state the following proposition. For $\Phi_{\bf k}=\phi_{k_1}\wedge\ldots\wedge\phi_{k_d}$, if $\phi_{i}$ appears in $\Phi_{\bf k}$, $\partial_{i}^R(\Phi_{\bf k})=(-1)^{d-i}\phi_{k_1}\wedge\ldots\wedge\hat{\phi}_{i}\wedge\ldots\wedge\phi_{k_d}$. If $\phi_i$ does not appear in $\Phi_{\bf k}$, $\partial_{i}^R(\Phi_{\bf k})=0$. The {\it left partial derivative operator} $\partial_i^L:\bigwedge^*(V(n))\rightarrow\bigwedge^*(V(n))$ is defined in the similar way. If $\phi_{i}$ appears in $\Phi_{\bf k}$, $\partial_{i}^L(\Phi_{\bf k})=(-1)^{i-1}\phi_{k_1}\wedge\ldots\wedge\hat{\phi}_{i}\wedge\ldots\wedge\phi_{k_d}$. If $\phi_i$ does not appear in $\Phi_{\bf k}$, $\partial_{i}^L(\Phi_{\bf k})=0$. It is easy to see that $\partial_i^{R}=(-1)^{d-1}\partial^L_i$ on $\bigwedge^d(V(n))$.
\end{defn}


\begin{prop}
  The decreasing operators are the right partial derivative operators on $\bigwedge ^*(V(n))$: $\phi_r^-\cdot\Phi_{\bf k}=\partial_{n-r-1}^R(\Phi_{\bf k})$.
\end{prop}
\begin{proof}
 What we want is to compute $\phi^-_r\cdot \Phi_{\bf k}$. Based on formula (\ref{decre-formula-ture}), we have
 \begin{equation}
\begin{split}
  (\phi^-_r\cdot \Phi_{\bf k})(x_1,\ldots,x_{d-1})&=(-1)^{|\lambda|+n-d}(\phi_r^+\cdot\Phi_{\bf k'})(x_{d},\ldots,x_n)\\
  &=(-1)^{|\lambda|+n-d}(\phi_r\wedge\phi_{k'_1}\wedge\ldots\wedge\phi_{k'_{n-{d}}})(x_{d},\ldots,x_n).
\end{split}
 \end{equation}

If $\phi_{n-r-1}$ is not in the $\Phi_{\bf k}$, $\phi_r$ will appear in $\Phi_{\bf k'}$. Thus $\phi_r^-\cdot\Phi_{\bf k}(x_1,\ldots,x_{d-1})=(\phi_r\wedge\phi_{k'_1}\wedge\ldots\wedge\phi_r\wedge\ldots\wedge\phi_{k'_{n-d}})(x_{d},\ldots,x_{n})=0$.

If $\phi_{n-r-1}$ appears in $\Phi_{\bf k}=\phi_{k_1}\wedge\ldots\wedge\phi_{k_d}$, $\phi_r$ won't be in $\Phi_{\bf k'}=\phi_{k'_1}\wedge\ldots\wedge\phi_{k'_{n-d}}$. Assume $k'_s< r<k'_{s+1}$. We have
\begin{equation}
\phi_r\wedge\phi_{k'_1}\wedge\ldots\wedge\phi_{k'_{n-{d}}}=(-1)^{s}\phi_{k'_1}\wedge\ldots\wedge\phi_{k'_s}\wedge\phi_r\wedge\phi_{k'_{s+1}}\wedge\ldots\wedge\phi_{k'_{n-{d}}}.
\end{equation}
We have to change this back to the ordinary presentation. First, let's find the partition associated to this polynomial. The index ${\bf l'}=(l'_{1},\ldots,l'_{n-d+1})$ is given by
\begin{equation}
  l'_i=\begin{cases}
    k'_{i-1}&s+2\leq i\leq n-d+1,\\
     r&i=s+1,\\
    k'_i&1\leq i\leq s.
  \end{cases}
  \end{equation}
Then the new partition $\mu'=(\mu'_{n-d+1},\ldots,\mu'_1)$ is given by
\begin{equation}
  \mu'_i=\begin{cases}
     \lambda'_{i-1}-1&s+2\leq i\leq n-d+1,\\
     r-s&i=s+1,\\
    \lambda'_i&1\leq i\leq s.
  \end{cases}
\end{equation}

Next step is to recover the partition $\mu$ from its transpose $\mu'$. From the definition of transpose partition, $\mu'_j=\#\{\mu_i\geq n-d+2-j\}$ for $1\leq j\leq n-d-1$. Then

   \begin{equation}
     \mu_{d-i}=\begin{cases}
   n-d-j  &\text{if}\ \lambda'_{j}\leq i\leq\lambda'_{j+1}-1\ \text{and}\ s+1\leq j\leq n-d-1,\\
    n-d-s &\text{if}\ r-s+1\leq i\leq \lambda'_{s+1}-1,\\
    n-d+1-s &\text{if}\ \lambda'_s+1\leq i\leq r-s,\\
    n-d+1-j&\text{if}\ \lambda'_j+1\leq i\leq \lambda'_{j+1}\ \text{and}\ 2\leq j\leq s-1,\\
    n-d+1&\text{if}\ 1\leq i\leq \lambda'_1.
     \end{cases}
   \end{equation}

By comparing it with

\begin{equation}
  \lambda_{d-i+1}=\begin{cases}
    n-d-j&\text{if}\ \lambda'_j+1\leq i\leq \lambda'_{j+1}\ \text{and}\ 2\leq j\leq n-d-1,\\
    n-d&\text{if}\ 1\leq i\leq \lambda'_1,
  \end{cases}
\end{equation}
we notice that $\mu_{i}=\lambda_{i+1}+1$ for $d-r+s\leq i\leq d-1$ and $\mu_{i}=\lambda_i$ for $1\leq i\leq d-r+s-1$.

Therefore, since $l_i=\mu_i+i-1$ for $1\leq i\leq d-1$ and $k_j=\lambda_j+j-1$ for $1\leq j\leq d$, it is easy to see that $l_i=k_{i+1}$ for $d-r+s\leq i\leq d-1$ and $l_{i}=k_i$ for $1\leq i\leq d-r+s-1$. Thus the resulted presentation is $(-1)^{n-d+s+|\lambda|+|\mu|}\phi_{k_1}\wedge\ldots\wedge\hat{\phi}_{n-r-1}\wedge\ldots\wedge\phi_{k_d}$$=$$(-1)^{r+s}\phi_{k_1}\wedge\ldots\wedge\hat{\phi}_{n-r-1}\wedge\ldots\wedge\phi_{k_d}$$=$$\partial_{n-r-1}^R(\Phi_{\bf k})$, which is $\Phi_{\bf k}$ applied by the right partial derivative of $\phi_{n-r-1}$. If $r<k'_1$ or $r>k'_{n-d}$, the similar process will lead to the same result.
\end{proof}


\subsection{Twisted decreasing operators}
From the above computations, it is obvious to have the following proposition about the twisted decreasing operators.

\begin{prop}
The twisted decreasing operators are the left partial derivative operators on  $\bigwedge ^*(V(n))$: $\hat{\phi}_r^-\cdot\Phi_{\bf k}=\partial_{n-r-1}^L(\Phi_{\bf k})$.
\end{prop}


\section{The double of representations}

\subsection{The double of untwisted representations}

Let $V(n)$ be the $n$-dimensional vector space spanned by $\{\phi_i \}_{i=0}^{n-1}$. The increasing and decreasing representations can be realized as creation operators $\{\alpha_i^+\}_{i=0}^{n-1}$ and annihilation operators $\{\alpha_i^-\}_{i=0}^{n-1}$ on $\bigwedge^*(V(n))$. Here $\alpha_i^+=\phi_i^+$ is the left wedge product, and $\alpha_i^-=\phi_{n-i-1}^-$ is the right partial derivative $\partial_i^R$.

Define $H=[\alpha_0^+,\alpha_0^-]$ and the following operators for $0\leq i\leq n-1$:
\begin{equation*}
  T_i=\frac{\alpha_i^++[H,\alpha_i^+]/2}{2},\quad  S_i=\frac{\alpha_i^--[H,\alpha_i^-]/2}{2}.
\end{equation*}

Then define the following operators
\begin{equation*}
\begin{split}
E_{0}&=-\frac{\alpha_0^-+[H,\alpha_0^-]/2}{2}, \quad F_0=\frac{\alpha_0^+-[H,\alpha_0^+]/2}{2},\\
E_1&=S_0,\quad F_1=T_0,\\
  E_i&=[T_{i-2},S_{i-1}],\quad F_i=[T_{i-1},S_{i-2}], \quad \text{for}\ \ 2\leq i\leq n,\\
  H_i&=[E_i,F_i],\quad \text{for}\ \ 0\leq i\leq n.
\end{split}
  \end{equation*}

  In the following, let $P_k$ be an arbitrary degree $k$ monomial in $\bigwedge^*(V(n))$. Denote by $R_i^j$ the operator which change the factor $\phi_i$ in $P_k$ to $\phi_j$.

\begin{lemma}\label{want-to-prove-matrix}
 For $2\leq i\leq n$,
  \begin{enumerate}
    \item $H(P_k)=(-1)^{k-1}P_k$.
    \item $E_0(P_k)=-\partial_0^R (P_k)$ if $k$ is even, and $\phi_0$ is included in $P_k$. Otherwise it's 0.

   \item $F_0(P_k)=\phi_0\wedge P_k$ if $k$ is odd, and $\phi_0$ is NOT included in $P_k$. Otherwise it's 0.

    \item $E_1(P_k)=\partial_0^R (P_k)$ if $k$ is odd, and $\phi_0$ is included in $P_k$. Otherwise it's 0.

   \item  $F_1(P_k)=\phi_0\wedge P_k$ if $k$ is even, and $\phi_0$ is NOT included in $P_k$. Otherwise it's 0.
     \item $S_{i-1}(P_k)=\partial_{i-1}^R (P_k)$  if $k$ is odd, and $\phi_{i-1}$ is included in $P_k$. Otherwise it's 0.

    \item $T_{i-1}(P_k)=\phi_{i-1}\wedge P_k$ if $k$ is even, and $\phi_{i-1}$ is NOT included in $P_k$. Otherwise it's 0.

    \item $E_i(P_k)=R_{i-1}^{i-2}(P_k)$ if $\phi_{i-1}$ is included in $P_k$ and $\phi_{i-2}$ is NOT. Otherwise it's 0.

   \item $F_i(P_k)=R_{i-2}^{i-1}(P_k)$ if $\phi_{i-2}$ is included in $P_k$ and $\phi_{i-1}$ is NOT. Otherwise it's 0.
    \item $H_0(P_k)=\begin{cases}
      -P_k \quad &\text{$k$ is even and $\phi_0$ is included in $P_k$}\\
      P_k &\text{$k$ is odd and $\phi_0$ is NOT included in $P_k$}\\
      0&otherwise
    \end{cases}$.

 \item   $H_1(P_k)=\begin{cases}
      P_k \quad &\text{$k$ is even and $\phi_0$ is NOT included in $P_k$}\\
      -P_k &\text{$k$ is odd and $\phi_0$ is included in $P_k$}\\
      0&otherwise
    \end{cases}$.

   \item $H_i(P_k)=\begin{cases}
      -P_k \quad &\text{$\phi_{i-1}$ is included in $P_k$ and $\phi_{i-2}$ is NOT included}\\
      P_k &\text{$\phi_{i-2}$ is included in $P_k$ and $\phi_{i-1}$ is NOT included}\\
      0&otherwise
    \end{cases}$.
  \end{enumerate}
\end{lemma}

\begin{proof}
   The proof of the lemma is straightforward.
\end{proof}

The main theorem below implies that the combination of two representations $R^+_n$ and $R^-_n$ of $A_1$-COHA forms an $D_{n+1}$-Lie algebra.

\begin{thm}
  The above operators satisfy the Serre relations for $0\leq i, j\leq n$:
  \begin{enumerate}
    \item $[H_i,H_j]=0$,
    \item $[E_i,F_j]=\delta_{ij}H_i$,
    \item $[H_i,E_j]=a_{ji}E_j,\quad [H_i,F_j]=-a_{ji}F_j$,
    \item $(adE_i)^{-a_{ji}+1}(E_j)=0$, if $i\neq j$,
    \item $(adF_i)^{-a_{ji}+1}(F_j)=0$, if $i\neq j$,
  \end{enumerate}
  where $(a_{ij})$ is the Cartan matrix for $D_{n+1}$-Lie algebras.
\end{thm}

\begin{proof}
The first statement holds since each $H_i$ is diagonal by Lemma \ref{want-to-prove-matrix}. The second is due to the definition of $H_i$ for $\delta_{ij}=1$. For the other relations, we need to check the following relations, which can be easily solved by Lemma \ref{want-to-prove-matrix}:
  \begin{enumerate}
  \item $a_{ii}=2$, for $0\leq i\leq n$,
  \item $a_{21}=a_{20}=a_{12}=a_{02}=a_{i-1,i}=a_{i,i-1}=-1$ for $3\leq i\leq n$,
  \item $a_{10}=a_{01}=a_{0,i}=a_{i,0}=a_{1,i}=a_{i,1}=a_{i,j}=a_{j,i}=0$, for $3\leq i\leq n$, $2\leq j\leq n$ and $|i-j|>1$

 \item $[E_0,F_1]=[E_1,F_0]=[E_0,F_j]=[E_1,F_j]=[E_i,F_0]=[E_i,F_1]=[E_i,F_j]=0$, for $2\leq i\neq j\leq n$,

 \item $[E_2,[E_2,E_0]]=[E_0,[E_0,E_2]]=[F_2,[F_2,F_0]]=[F_0,[F_0,F_2]]=0$,

  \item $[E_{i-1},[E_{i-1},E_i]]=[E_{i},[E_{i},E_{i-1}]]=[F_{i-1},[F_{i-1},F_i]]=[F_{i},[F_{i},F_{i-1}]]=0$, for $2\leq i\leq n$,

\item $[E_0,E_1]=[E_0,E_i]=[E_1,E_i]=[E_i,E_j]=0$ for $3\leq i\leq n$, $2\leq j\leq n$ and $|i-j|>1$,
\item $[F_0,F_1]=[F_0,F_i]=[F_1,F_i]=[F_i,F_j]=0$ for $3\leq i\leq n$, $2\leq j\leq n$ and $|i-j|>1$.
  \end{enumerate}
\end{proof}

\subsection{The double of twisted representations}
Use the setting from the previous subsection. Let  $\hat{\alpha}_i^-=\hat{\phi}_{n-i-1}^-$ be the left partial derivative $\partial_i^L$. Now we use creation operators $\{\alpha_i^+\}_{i=0}^{n-1}$ and twisted annihilation operators $\{\hat{\alpha}_i^-\}_{i=0}^{n-1}$ to form representations. These relations show that the double of twisted representations form a finite Clifford algebra.
\begin{thm}
Operators $\{\alpha_i^{+}\}_{i=0}^{n-1}$ and $\{\hat{\alpha}_i^{-}\}_{i=0}^{n-1}$ satisfy the following relations:
\begin{enumerate}
\item $\alpha^+_i\alpha^+_j+\alpha^+_j\alpha^+_i=0$,
\item $\hat{\alpha}^-_i\hat{\alpha}^-_j+\hat{\alpha}^-_j\hat{\alpha}^-_i=0$,
\item $\alpha^+_i\hat{\alpha}^-_j+\hat{\alpha}^-_j\alpha^+_i=\delta_{i,j}$.
\end{enumerate}
\end{thm}
\begin{proof}
The proof is straightforward by applying the formula in the definitions of the operators to the basis vectors of $\bigwedge^*(V(n))$.
\end{proof}

\section{Further discussions}
For fixed $n$, the double of $R^+_n$ and $R^-_n$ forms $D_{n+1}$-Lie algebra, and the double of $R^+_n$ and $\hat{R}^-_n$ forms a finite Clifford algebra. This leads to the following conjecture stated in \cite{Soi2014}.
\begin{conj}\cite{Soi2014}
  Full COHA for the quiver $A_1$ is isomorphic to the infinite Clifford algebra $Cl_c$ with generators $\phi_n^{\pm},\ n\in\fZ$ and the central element $c$, subject to the standard anticommuting relations between $\phi_n^+$ (resp. $\phi_n^-$) as well as the relation $\phi_n^+\phi_m^-+\phi_m^-\phi_n^+=\delta_{n,m}c$.
\end{conj}

\begin{remark}
  As stated in \cite{Soi2014}, in the case of finite-dimensional representations we have $c\mapsto0$ and we see two representations of the infinite Grassmann algebra, which are combined in the representations of the orthogonal Lie algebra.
\end{remark}

 \subsection*{Acknowledgement}
I thank to Yan Soibelman who introduced me to this subject, stated the problem and made multiple comments on the draft of this paper. I thank Hans Franzen for helpful communications, whose paper helped me to simplify my proof. I also thank to Zongzhu Lin, Zhaobin Fan, Jie Ren and Hui Chen for helpful discussions.

\appendix
\renewcommand{\thesection}{\Alph{chapter}.\arabic{section}}


\end{document}